\documentclass[11pt,reqno,a4paper]{amsart}
\usepackage{amsmath,amssymb,amsthm,amsaddr}
\usepackage{enumitem}
\usepackage[margin=3cm,top=3.8cm,bottom=3.5cm,footskip=1cm,headsep=1.5cm]{geometry}
\usepackage[utf8]{inputenc}
\usepackage{amsthm}
\usepackage[british]{babel}

\usepackage{accents}
\makeatletter
\newcommand{\sbullet}{%
  \hbox{\fontfamily{lmr}\fontsize{.25\dimexpr(\f@size pt)}{0}\selectfont\textbullet}}
\DeclareRobustCommand{\mathbullet}{\accentset{\sbullet}}
\makeatother

\author{H. Egger and M. Sabouri}
\address{Department of Mathematics, TU Darmstadt, Germany}
\email{egger@mathematik.tu-darmstadt.de}
\email{sabouri@mathematik.tu-darmstadt.de}

\title[High order approximation for poroelasticity]{On the structure preserving high-order\\ approximation of quasistatic poroelasticity}

\newtheorem{theorem}{Theorem}

\newtheorem{lemma}[theorem]{Lemma}
\newtheorem{assumption}[theorem]{Assumption}
\newtheorem{problem}[theorem]{Problem}

\theoremstyle{definition}
\newtheorem{remark}[theorem]{Remark}

\def\dt{\partial_t}

\def\eps{\epsilon}
\def\div{\textrm{div}}

\def\RR{\mathbb{R}}

\def\V{\mathcal{V}}
\def\Q{\mathcal{Q}}

\def\Th{\mathcal{T}_h}

\def\rmA{\mathtt{A}}
\def\rmB{\mathtt{B}}
\def\rmK{\mathtt{K}}
\def\rmS{\mathtt{S}}
\def\rmu{\mathtt{u}}
\def\rmc{\mathtt{c}}
\def\rmp{\mathtt{p}}
\def\rmf{\mathtt{f}}
\def\rmg{\mathtt{g}}

\begin{document}

\begin{abstract}
We consider the systematic numerical approximation of Biot's quasistatic model for the consolidation of a poroelastic medium. Various discretization schemes have been analysed for this problem and inf-sup stable finite elements have been found suitable to avoid spurios pressure oscillations in the initial phase of the evolution. 
In this paper, we first clarify the role of the inf-sup condition for the well-posedness of the continuous problem and discuss the choice of appropriate initial conditions. 
We then develop an abstract error analysis that allows us to analyse some approximation schemes discussed in the literature in a unified manner. 
In addition, we propose and analyse the high-order time discretization by a scheme that can be interpreted as a variant of continuous-Galerkin or particular Runge-Kutta methods applied to a modified system. The scheme is designed to preserve both, the underlying differential-algebraic structure and energy-dissipation property of the problem.
In summary, we obtain high-order Galerkin approximations with respect to space and time and derive order-optimal convergence rates.  
The numerical analysis is carried out in detail for the discretization of the two-field formulation by Taylor-Hood elements and a variant of a Runge-Kutta time discretization. Our arguments can however be extended to three- and four field formulations and other time discretization strategies. 
\end{abstract}

\maketitle

\noindent {\bf Keywords}. 
Galerkin  approximation, mixed finite elements, structure preserving discretization, differential-algbraic equations, Biot system, poroelasticity\\

\section{Introduction}

In linear quasistatic theory \cite{Biot41} the consolidation of a poroelastic solid, which is fully saturated by an incompressible fluid, is usually described by Biot's equations 
\begin{align}
-\div (2 \mu \eps(u) + \lambda \div(u) I) + \alpha \nabla p  &=  f, \label{eq:biot1}\\
\alpha \div (\mathbullet u) - \div(\kappa \nabla p) &= g, \label{eq:biot2}
\end{align}
together with appropriate initial and boundary conditions.
Here $u$ is the solid displacement of the porous medium and $p$ is the pressure of the residing fluid whereas $f$ denotes the density of external forces and $g$ is the fluid source density. 
The Biot parameter $\alpha$ is usually close to one and the hydraulic conductivity $\kappa$ is assumed strictly positive. 

Due to its many applications, e.g., in geosciences or biomathematics, the theoretical and numerical analysis of \eqref{eq:biot1}--\eqref{eq:biot2} has attracted significant interest in the mathematical literature. 
The existence of unique solutions for the Biot system has been established by Zenisek \cite{Zenisek84a,Zenisek84} under some regularity conditions on the data via discretization with finite elements and the implicit Euler method using a-priori estimates and compactness arguments. 
Showalter \cite{Showalter00} established well-posedness for different formulations of the Biot model by the method of semi-groups. 
Murad and Loula \cite{MuradLoula92,MuradLoula94} investigated the Galerkin approximation by stable and unstable finite element pairs and established decay estimates for the discrete error via improved energy estimates.
Mixed finite element approximations of the three-field formulation in which the seapage velocity $v=-\kappa \nabla p$ is introduced as a new variable,  were considered by Phillips and Wheeler \cite{PhillipsWheeler07a,PhillipsWheeler07b,PhillipsWheeler08}. 
Yi \cite{Yi13} investigated the discretization of a four-field formulation in which also the elastic stress field $\sigma = 2 \mu \eps(u) + \lambda \div(u) I$ is introduced additionally.
In \cite{KanschatRiviere18} Kanschat and Riviere considered the approximation of the three-field formulation with a non-conforming approximation of the elastic deformation by $H(\div)$ finite elements. 
While most of the previous papers only utilized low order approximations in time, Bause et al. \cite{Bause17,Bause18p} considered the efficient implementation of high order time approximations for poroelasticity, but no convergence analysis was conducted. 
Various available results concern the discretization of the static Biot systems that arise after time discretization, see e.g., \cite{HongKraus18,OyarzuaRuizBaier16} in which stability with respect to model and discretization parameters has been studied.
It is well-accepted by now \cite{MuradLoula94,PhillipsWheeler08,Yi17}, that inf-sup stable finite element pairs should be used to avoid spurious pressure oscillations that might appear in the initial phase of the simulations.

\medskip 

In this paper, we consider the systematic construction of high-order approximations for the quasistatic system \eqref{eq:biot1}--\eqref{eq:biot2} by Galerkin methods in space and time. 
We give a short and concise proof of well-posedness of the continuous problem which motivates our functional analytic setting and provides guidelines for the choice of initial conditions. The regularity requirements for the data are based on the physically relevant energy-dissipation structure of the evolution problem and compatibility conditions for the initial values are derived from the differential-algebraic structure.    
We then consider the Galerkin approximation in space and discuss the properties of the resulting differential-algebraic equations. In particular, we show that the index of the differential-algebraic system \cite{BrenanCampbellPetzold,KunkelMehrmann} is one, independent of the approximation spaces, while a discrete inf-sup condition, i.e., the surjectivity of the discrete divergence operator, is required to guarantee the well-posedness of the semi-discretization under natural compatibility conditions for the initial conditions. Otherwise, additional non-physical constraints for the initial conditions arise.
Afterwards, we  show that the discrete error, i.e., the difference between the Galerkin semi-discretization and an appropriate elliptic projection, only depends on the approximation error in the displacement $u$ which implies a certain super-convergence for the discrete error. 
We then consider the time-discretization by a strategy which is capable of preserving both, the particular differential-algebraic structure and the energy-dissipation property of the underlying problem.
The resulting scheme can be interpreted as a continuous Galerkin method or a variant of certain Runge-Kutta methods applied to a reformulation of the problem in which the algebraic equation is differentiated. The latter interpretation also allows us for an efficient implementation. 

\medskip 

The remainder of the manuscript is organized as follows: 
In Section~\ref{sec:wellposed}, we introduce an abstract evolution problem which covers the weak formulation of the Biot system \eqref{eq:biot1}--\eqref{eq:biot2} as a special case, and we establish its well-posedness under mild assumptions on the problem structure and the data. A natural compatibility condition for the initial conditions and the underlying energy-dissipation structure are presented together with the resulting a-priori bounds. 
In Section~\ref{sec:galerkin}, we discuss the Galerkin approximation of the abstract problem and investigate the effect of the discrete inf-sup stability on the index of the resulting differential-algebraic system. We then establish an abstract convergence result for inf-sup stable approximations. 
In Section~\ref{sec:time}, we propose a time discretization scheme which can be interpreted as an inexact continuous Galerkin approximation or a variant of certain Runge-Kutta collocation methods.
In the spirit of \cite{Akrivis11}, this method can be interpreted in a pointwise sense which allows us to show that it preserves the underlying energy-dissipation structure and to conduct the error analysis with similar arguments as on the continuous level. 
In Section~\ref{sec:biot}, we apply the results to the discretization of the Biot system by Taylor-Hood finite elements and derive order optimal error estimates in space and time. 
Numerical tests are presented in Section~\ref{sec:num} for illustration of our theoretical results, 
and some directions for possible extensions and further investigations are discussed in the last section.

\section{An abstract model problem} \label{sec:wellposed}

We now introduce an abstract model problem which covers the Biot system as a special case, and then establish its well-posedness under mild structural assumptions on the data. 
In addition, we briefly discuss the regularity of the solution as well as the underlying energy-dissipation identity which us later used to establish a-priori estimates.

\subsection{Notation and summary of results}

Let $\V$, $\Q$, $\Q_0$ be real Hilbert spaces with compact embedding $\Q \subset \Q_0$. By identifying $\Q_0$ with its dual $\Q_0^*$, we obtain a Gelfand-triple $\Q \subset \Q_0 = \Q_0^* \subset \Q^*$; see \cite[Chapter~XVIII]{DL5}. The symbol $\langle \cdot , \cdot \rangle$ will be used to denote the duality product on $\V^* \times \V$, $\Q^* \times \Q$, and $\Q_0^* \times \Q_0$.

In the following, we consider an abstract evolution problem stated in weak form as
\begin{align}
a(u(t),v) - b(v,p(t)) &= \langle f(t), v\rangle, \qquad \forall v \in \V, \ t>0, \label{eq:var1}\\
b(\mathbullet u(t), q) + k(p(t),q) &= \langle g(t), q\rangle, \qquad \forall q \in \Q, \ t>0. \label{eq:var2} 
\end{align}
Here $a : \V \times \V \to \RR$, $k : \Q \times \Q \to \RR$, and $b : \V \times \Q_0 \to \RR$ 
are given bilinear forms and $f$, $g$ are functions of time with values in $\V^*$ and $\Q^*$, respectively. To establish the well-posedness for the corresponding Cauchy problem, we make the following assumptions.
\begin{assumption} \label{ass:main}
\begin{itemize} \itemindent-0em
 \item[(A1)] $a: \V \times \V \to \RR$ is bounded, elliptic, and symmetric.
 \item[(A2)] $k: \Q \times \Q \to \RR$ is bounded, elliptic, and symmetric. 
 \item[(A3)] $b: \V \times \Q_0 \to \RR$ is bounded and inf-sup stable; see below.
\end{itemize}
\noindent
We further use $a(\cdot,\cdot)$ and $k(\cdot,\cdot)$ as scalar products on $\V$ and $\Q$, and thus obtain 
\begin{align}
a(u,v) \le \|u\|_\V \|v\|_\V \qquad \text{and} \qquad a(v,v) = \|v\|_\V^2,\\
k(p,q) \le \|p\|_\Q \|q\|_\Q \qquad \text{and} \qquad k(q,q) = \|q\|_\Q^2.
\end{align}
The norm of $\Q_0$ is chosen such that $\|q\|_{\Q_0} \le \|q\|_\Q$ and (A3) yields
\begin{align} 
b(v,q) \le C_b \|v\|_\V \|q\|_{\Q_0} \qquad \text{and} \qquad \sup_{\|v\|_\V=1} b(v,q) \ge \beta \|q\|_{\Q_0},
\end{align}
for some $\beta,C_b > 0$. All estimates hold uniformly with respect to their arguments.
\end{assumption}
The following theorem summarizes the basic results about well-posedness and regularity of the solution for the above system with appropriate initial conditions. 
\begin{theorem} \label{thm:wellposed}
Let Assumption~\ref{ass:main} hold. Then for any $f \in H^1(0,T;\V^*)$, $g \in L^2(0,T;\Q^*)$, and $p_0 \in \Q_0$, there exists a unique weak solution 
\begin{align*}
u \in C(0,T;\V) \quad \text{ with } \quad B u \in H^1(0,T;\Q^*), \\
p \in L^2(0,T;\Q) \cap H^1(0,T;\Q^*) \cap C(0,T;\Q_0),  
\end{align*}
which satisfies \eqref{eq:var1}--\eqref{eq:var2} for a.e. $0<t<T$ and the initial conditions 
$p(0)=p_0$, $u(0)=u_0$ where $u_0 \in \V$ is defined by the compatibility conditions
\begin{align} \label{eq:compatibility}
a(u_0,v) - b(v,p_0) &= \langle f(0),v\rangle, \quad \forall v \in \V. 
\end{align}
If, in addition, $g \in L^2(0,T;\Q_0^*) \cup H^1(0,T;\Q^*)$ and $p_0 \in \Q$,
then 
\begin{align*}
p \in L^\infty(0,T;\Q) \cap H^1(0,T;\Q_0) 
\qquad \text{and} \qquad 
u \in H^1(0,T;\V).
\end{align*}
In both cases, the solution can be bounded by the data in the natural norms.
\end{theorem} 

\noindent
The remainder of this section is devoted to a proof of these assertions.

\subsection{Well-posedness}

As usual, we associate via $\langle A u,v\rangle=a(u,v)$, $\langle K p,q\rangle = k(p,q)$, and $\langle B v,q\rangle=b(v,q)=\langle B^* q, v\rangle $ to any of the bilinear forms a linear operator
$A:\V \to \V^*$, $K:\Q \to \Q^*$, $B:\V \to \Q_0^*$ and its adjoint $B^* : \Q_0 \to \V^*$.  
This allows us to rewrite \eqref{eq:var1}--\eqref{eq:var2} as equivalent operator equations
\begin{align}
A u(t) - B^* p(t) &= f(t), \qquad t>0, \label{eq:dae1}\\
B \mathbullet u(t) + K p(t) &= g(t), \qquad t>0, \label{eq:dae2}
\end{align}
which have to be understood as equations in the sense of $\V^*$ and $\Q^*$, respectively.
From the properties of the bilinear forms, we immediately conclude the following result.
\begin{lemma} \label{lem:operators}
Let Assumption~\ref{ass:main} hold. Then 
\begin{itemize}  \itemindent-1em
\item[(i)] $A: \V \to \V^*$ and $K:\Q \to \Q^*$ are symmetric, bounded, and boundedly invertible\\ \hspace*{-1em}with $\|A v\|_{\V^*} \le \|v\|_\V$, $\|A^{-1} f\|_{\V} \le \|f\|_{\V^*}$
and $\|K q\|_{\Q^*} \le \|q\|_\Q$, $\|K^{-1} g\|_{\Q} \le \|g\|_{\Q^*}$.
\item[(ii)] $B : \V \to \Q_0^*$ is bounded, surjective with closed range, and
$\|B v\|_{\Q_0^*} \le C_b \|v\|_\V$. 
\item[(iii)] $B^* : \Q_0 \to \V^*$ is bounded and injective with 
$\beta \|q\|_{\Q_0} \le \|B^* q\|_{\V^*} \le C_b \|q\|_{\Q_0}$.
\end{itemize}
\end{lemma}
Using property (i) of the previous lemma, we can eliminate $u$ from \eqref{eq:dae1} leading to
\begin{align} \label{eq:u}
u(t) &= A^{-1} (f(t) + B^* p(t)). 
\end{align}
Inserting this into equation \eqref{eq:dae2} yields the Schur complement problem
\begin{align} \label{eq:dae3}
B A^{-1} B^* \mathbullet p(t) + K p(t) &= g(t) - B A^{-1} \mathbullet f(t). 
\end{align}
From the properties of the operators in Lemma~\ref{lem:operators}, we deduce the following assertions.
\begin{lemma} \label{lem:oprators2}
Let Assumption~\ref{ass:main} and the conditions on the data in Theorem~\ref{thm:wellposed} hold. Then 
\begin{itemize} \itemindent-1em
 \item[(v)] The operator $B A^{-1} B^* : \Q_0 \to \Q_0^*$ is symmetric, bounded, and elliptic.
 \item[(vi)] $h = g - B A^{-1} \mathbullet f =: h_0 + h_1 \in L^2(0,T;\Q^*) \cup L^2(0,T;\Q_0^*)$.
\end{itemize}
\end{lemma}
Note that, according to property (v), the operator $C := B A^{-1} B^*: \Q_0 \to \Q_0^*$ induces a symmetric, bounded, and elliptic bilinear form $c: \Q_0 \times \Q_0 \to \RR$, 
and the operator equation \eqref{eq:dae3} can hence be written in an equivalent weak form as
\begin{align} \label{eq:var3}
c(\mathbullet p(t),q) + k(p(t),q) &= \langle h(t), q\rangle, \qquad \forall q \in \Q, \ t>0. 
\end{align}
This is an abstract parabolic equation whose well-posedness can be proven via Galerkin approximation; see \cite[Chapter~7]{Evans} or \cite[Chapter~XVIII]{DL5}. We thus obtain
\begin{lemma} \label{lem:wellposed_p}
Let Assumption~\ref{ass:main} hold. 
Then for any $p_0 \in \Q_0$ and $h \in L^2(0,T;\Q^*)$, 
the reduced evolution problem \eqref{eq:var3} has a unique weak solution 
\begin{align}
p \in L^2(0,T;\Q) \cap H^1(0,T;\Q^*)
\end{align}
with initial value $p(0)=p_0$. Moreover, $p \in C(0,T;\Q_0)$ by embedding. 
\end{lemma}

This proves existence and uniqueness of a solution $p$ for problem \eqref{eq:var3} as well as a-priori estimates in the corresponding norms. 
Inserting $p$ into \eqref{eq:u} and using the mapping properties of the operators $A^{-1}$ and $B^*$ yields $u \in C(0,T;\V)$. 
The compatibility condition \eqref{eq:compatibility} then follows from continuity.  
Furthermore from \eqref{eq:dae2}, one can see that $B\mathbullet{u} \in L^2(0,T;\Q^*)$. 
By equivalence of \eqref{eq:dae1}--\eqref{eq:dae2} with the variational form, we obtain existence and uniqueness of a weak solution $(u,p)$ for the system \eqref{eq:var1}--\eqref{eq:var2}
with the given initial conditions.
Moreover, we have established the a-priori bounds in the first part of Theorem~\ref{thm:wellposed}.

\subsection{Regularity}

We now turn to the improved a-priori estimates stated in Theorem~\ref{thm:wellposed}.
Additional regularity of the solution to the reduced problem \eqref{eq:var3} can be obtained with similar arguments as in \cite[Chapter~7]{Evans}.
Assume that $f \in H^2(0,T;\V^*)$ and $g=g_1+g_2$ with $g_1\in L^2(0,T;\Q_0^*)$ and $g_2 \in H^1(0,T;\Q^*)$.
Then from the properties of the operators $A$ and $B$, we deduce that $h=g-B A^{-1} \mathbullet f = (g_1 - B A^{-1}\mathbullet f) + g_2 =: h_1 + h_2$ with $h_1\in L^2(0,T;\Q_0^*)$ and $h_2 \in H^1(0,T;\Q^*)$.
By formally testing \eqref{eq:var3} with $q=\mathbullet p$, we  obtain 
\begin{align*}
c(\mathbullet p,\mathbullet p) + \frac{d}{dt} \frac{1}{2} k(p,\mathbullet p) 
&= \langle h,\mathbullet p \rangle,
\end{align*}
and by integrating with respect to time, we further get
\begin{align*}
\int_0^t c(\mathbullet p,\mathbullet p) dt &+ k(p(t),p(t)) 
\le k(p(0),p(0)) + \int_0^t \langle h_1, \mathbullet p \rangle - \langle \mathbullet h_2, p \rangle dt + \langle h_2,p\rangle|_0^t \\
&\le \frac{3}{2} k(p(0),p(0)) + C_1  \int_0^t \|\mathbullet f\|^2_{\V^*} + \|g_1\|^2_{\Q_0^*}  + \|\mathbullet g_2\|_{\Q^*}^2 dt \\
& \qquad + C_2 (\|g_2(0)\|^2_{\Q^*} + \|g_2(t)\|^2_{\Q^*}) 
+ \frac{1}{2} \int_0^t c(\mathbullet p,\mathbullet p) dt + \frac{1}{2} k(p(t),p(t)).
\end{align*}
Here we used $\gamma \|q\|_{\Q_0} \le c(q,q)$ and $c(q,q) \le C \|q\|^2_{\Q_0} \le C k(q,q)$, which follow directly from the properties of the bilinear forms and  norms.  
The last two terms in the above estimate can be absorbed into the left hand side, which yields the required estimates for $p$. Inserting this into \eqref{eq:u} leads to the improved regularity for $u$. 
This concludes the proof of Theorem~\ref{thm:wellposed}.

\subsection{Choice of initial conditions}
Before we proceed, let us discuss in more detail the choice of initial conditions. 
Under Assumptions~\ref{ass:main}, the variational problem 
\begin{align}
a(u_0,v)- b(v,p_0) &= \langle f_0,v\rangle, \quad \forall v \in \V, \label{eq:initial1}\\
b(u_0,q)           &= \langle \phi_0,q\rangle, \quad \forall q \in \Q, \label{eq:initial2}
\end{align}
is uniquely solvable for all $f_0 \in \V^*$ and $\phi_0 \in \Q^*$, which is a direct consequence of Brezzi's splitting lemma \cite{Brezzi74}. 
In the above arguments, we simply chose $f_0=f(0)$ and specified $p(0)=p_0 \in \Q_0$. By assumption (A1), we can  determine $u_0 \in \V$ from \eqref{eq:initial1}, and inserting $u_0$ into the equation \eqref{eq:initial2} determines $\phi_0 = B u_0 \in \Q_0^*$. 
Alternatively, one could set $f_0=f(0)$, choose $\phi_0 \in Q_0^* \simeq Q_0$ freely, and then determine $u_0$, $p_0$ by solving the coupled system \eqref{eq:initial1}--\eqref{eq:initial2}. 
This again provides initial values $u_0 \in \V$ and $p_0 \in \Q_0$ satisfying \eqref{eq:compatibility}, where $p_0$ can be used as initial value for \eqref{eq:dae3}. 
The following choices of initial conditions are therefore equivalently possible:
\begin{itemize}
\item[(i)] choose $p_0 \in \Q_0$ and determine $u_0$ by \eqref{eq:initial1};
\item[(ii)] choose $\phi_0 \in \Q_0^*$ and determine $u_0$, $p_0$ by \eqref{eq:initial1}--\eqref{eq:initial2}.
\end{itemize}
While the first choice is the natural one for the reduced problem \eqref{eq:var3}, the second choice seems more natural for the coupled system \eqref{eq:var1}--\eqref{eq:var2}. 
As indicated above, both choices are possible and actually equivalent in the considered functional analytic setting.

\subsection{Energy-dissipation}

The following property of the dynamical system will serve as the basic tool for the stability analysis of approximation schemes in later sections.
\begin{lemma}
Let Assumption~\ref{ass:main} hold and let $(u,p)$ denote a regular solution of \eqref{eq:dae1}--\eqref{eq:dae2} in the sense of Theorem~\ref{thm:wellposed}. Then  
\begin{align}
\frac{d}{dt} \frac{1}{2} a(u,u) + k(p,p) &= \langle f,\mathbullet u\rangle + \langle g,p\rangle.
\end{align}
\end{lemma}
\begin{proof}
Formal differentiation in time yields 
\begin{align*}
\frac{d}{dt} \frac{1}{2} a(u,u)
&= a(u,\mathbullet u) 
 = \langle f,\mathbullet u\rangle + b(\mathbullet u,p)
 = \langle f, \mathbullet u\rangle + \langle g,p\rangle - k(p,p),
\end{align*}
which after rearrangement of the terms already yields the assertion of the lemma. 
\end{proof}

\begin{remark}
By integration in time, the validity of the stability estimate can be extended to less regular weak solutions, and these estimates provide an alternative route for proving uniqueness and a-priori estimates for weak solutions. 
\end{remark}

\section{Galerkin approximation} \label{sec:galerkin}

For discretization of the variational problem \eqref{eq:var1}--\eqref{eq:var2},  
we now consider Galerkin approximations in finite dimensional sub-spaces $\V_h \subset \V$ and $\Q_h \subset \Q$, i.e., we search for semi-discrete functions $u_h \in H^1(0,T;\V_h)$ and $p_h \in H^1(0,T;\Q_h)$ satisfying 
\begin{align}
a(u_h(t),v_h) - b(v_h,p_h(t)) &= \langle f(t), v_h\rangle, \qquad \forall v_h \in \V_h, \ t>0, \label{eq:var1h}\\
b(\mathbullet u_h(t), q_h) + k(p_h(t),q_h) &= \langle g(t), q_h\rangle, \qquad \forall q_h \in \Q_h, \ t>0, \label{eq:var2h} 
\end{align}
together with appropriate initial conditions to be specified below. 
Using similar arguments as for the analysis on the continuous level, we will show the following result.
\begin{theorem} \label{thm:wellposedh}
Let Assumption~\ref{ass:main} hold. Then \eqref{eq:var1h}--\eqref{eq:var2h} is a regular system of differential-algebraic equations of index $1$. If $b$ is inf-sup stable on $\V_h \times \Q_h$, i.e., 
\begin{itemize}\itemindent0em
 \item[(A3h)] $\qquad \sup_{\|v_h\|_{\V} = 1} b(v_h,q_h) \ge \beta_h \|q_h\|_{\Q_{0}} \qquad \forall q_h \in \Q_h $ with some constant $\beta_h > 0$,
\end{itemize}
then a unique solution to the system \eqref{eq:var1h}--\eqref{eq:var2h} exists for any choice of initial values $u_h(0)=u_{h,0} \in \V_h$ and $p_h(0)=p_{h,0} \in \Q_h$ satisfying
\begin{align} \label{eq:compatibilityh}
a(u_{h,0},v_h)  - b(v_h,p_{h,0}) &= \langle f(0),v_h\rangle, \quad \forall v_h \in \V_h.
\end{align}
If (A3h) is not valid, additional compatibility conditions for $u_{h,0}$ and $p_{h,0}$ are required. 
\end{theorem}

\begin{remark}
A quick comparison with Theorem~\ref{thm:wellposed} shows that \eqref{eq:compatibilityh} are the natural compatibility conditions for the initial values of the problem under consideration,
while the additional conditions required when (A3h) is not valid are artificial; see below.
Inf-sup stable approximation spaces $\V_h$, $\Q_h$ are therefore required to guarantee the well-posedness of the semi-discrete problem without artificial conditions on the initial values.
\end{remark}

In the following, we give a detailed proof of the above theorem and then turn to the error analysis of the Galerkin approximations defined in the beginning of the section.

\subsection{Proof of Theorem~\ref{thm:wellposedh}}

Choice of a basis for the spaces $\V_h$, $\Q_h$ allows to convert the semi-discrete problem into an equivalent system of differential-algebraic equations
\begin{align} \label{eq:daesys}
\begin{pmatrix} 0 & 0 \\ \rmB & 0  \end{pmatrix}
\begin{pmatrix} \mathbullet \rmu \\ \mathbullet \rmp \end{pmatrix}
+ 
\begin{pmatrix} \rmA & -\rmB^\top \\ 0 & \rmK  \end{pmatrix}
\begin{pmatrix} \rmu \\ \rmp \end{pmatrix}
=
\begin{pmatrix} \rmf \\ \rmg \end{pmatrix},
\end{align}
with $\rmu,\rmp$ and $\rmf,\rmg$ denoting the coordinate vectors of the discrete solutions and data, 
and $\rmA$, $\rmB$, $\rmK$ being matrices of appropriate size.
From (A1)--(A2), one can see that the matrices $\rmA$ and $\rmK$ are symmetric and positive definite, and therefore, the matrix
\begin{align*}
\rmS(\lambda)=\begin{pmatrix} 0 & 0 \\ \rmB & 0 \end{pmatrix} 
+ \lambda  \begin{pmatrix} \rmA & -\rmB^\top \\ 0 & \rmK \end{pmatrix}, 
\end{align*}
is positive definite, e.g., for $\lambda=1$. As a consequence, the matrix pencil $\rmS(\lambda)$ is regular and so the system \eqref{eq:daesys} is a \emph{regular} differential-algebraic equation; see \cite{BrenanCampbellPetzold,KunkelMehrmann} for details. 
Differentiating the first equation in \eqref{eq:daesys} leads to 
\begin{align} \label{eq:daesys2}
\begin{pmatrix} \rmA & -\rmB^\top \\ \rmB & 0  \end{pmatrix}
\begin{pmatrix} \mathbullet \rmu \\ \mathbullet \rmp \end{pmatrix}
+ 
\begin{pmatrix} 0 & 0 \\ 0 & \rmK  \end{pmatrix}
\begin{pmatrix} \rmu \\ \rmp \end{pmatrix}
=
\begin{pmatrix} \mathbullet \rmf \\ \rmg \end{pmatrix}.
\end{align}
Note that the matrix in front of the time derivatives is regular if, and only if, $\rmB$ is surjective. 
In that case, \eqref{eq:daesys2} is an (implicit) ordinary differential equation, and the existence of a unique solution $(\rmu,\rmp)$ follows for any choice of initial conditions $\rmu(0)=\rmu_0$ and $\rmp(0)=\rmp_0$. 
By integration of the first equation, one obtains 
\begin{align*}
\rmA \rmu(t) - \rmB^\top \rmp(t) = \rmf(t) + \rmc,
\end{align*}
with $\rmc=\rmA \rmu(0) - \rmB^\top \rmp(0) - \rmf(0)$. Hence $(\rmu,\rmp)$ solves the original system \eqref{eq:daesys} if, and only if, $\rmc=0$, i.e., when the compatibility condition 
\begin{align} \label{eq:compat_la}
\rmA \rmu_0- \rmB^\top \rmp_0 = \rmf(0)
\end{align}
is satisfied.
If, on the other hand, $\rmB$ is not surjective, then \eqref{eq:daesys2} is still a differential-algebraic equation. By change of basis, we may transform the system into 
\begin{align*}
\begin{pmatrix}
\rmA   &    0 & -\rmB_2^\top \\
0   &    0 & 0 \\
\rmB_2 &    0 & 0
\end{pmatrix}
\begin{pmatrix}
\mathbullet \rmu \\ \mathbullet \rmp_1 \\ \mathbullet \rmp_2 
\end{pmatrix}
+ 
\begin{pmatrix}
0 &   0 &   0 \\
0 & \rmK_1 &   0 \\
0 &   0 & \rmK_2       
\end{pmatrix}
\begin{pmatrix}
\rmu \\ \rmp_1 \\ \rmp_2 
\end{pmatrix}
&=
\begin{pmatrix}
\mathbullet \rmf \\ \rmg_1 \\ \rmg_2 
\end{pmatrix},
\end{align*}
with $\rmB_2$ being surjective and $\rmK_1$, $\rmK_2$ both being positive definite.
Differentiation of the second equation then leads to the system
\begin{align}\label{eq:daesys3}
\begin{pmatrix}
\rmA   &   0 & -\rmB_2^\top \\
0   & \rmK_1 &    0 \\
\rmB_2 &   0 &    0
\end{pmatrix}
\begin{pmatrix}
\mathbullet \rmu \\ \mathbullet \rmp_1 \\ \mathbullet \rmp_2 
\end{pmatrix}
+ 
\begin{pmatrix}
0 &   0 &   0 \\
0 &   0 &   0 \\
0 &   0 & \rmK_2       
\end{pmatrix}
\begin{pmatrix}
\rmu \\ \rmp_1 \\ \rmp_2 
\end{pmatrix}
&=
\begin{pmatrix}
\mathbullet \rmf \\ \mathbullet \rmg_1 \\ \rmg_2 
\end{pmatrix}.
\end{align}
Using the fact that $\rmB_2$ is surjective and  $\rmA$ and $\rmK_1$ are positive definite, one can deduce that the matrix in front of the time derivatives is regular, hence \eqref{eq:daesys3} is an ordinary differential equation. An inspection of the right hand side shows that none of the equations has been differentiated more than once, and hence the index of \eqref{eq:daesys} is also one in this case. 
To obtain equivalence of \eqref{eq:daesys3} with \eqref{eq:daesys}, not only the compatibility condition \eqref{eq:compat_la} is required, but the additional artificial condition 
\begin{align} \label{eq:compat2_la}
\rmK_1 \rmp_1(0) = \rmg_{1}(0)
\end{align}
has to be enforced, which is only caused by the inappropriate numerical approximation. 

The assertions of Theorem~\ref{thm:wellposedh} follow immediately from the above results by equivalence of the differential-algebraic system \eqref{eq:daesys} 
with the weak formulation \eqref{eq:var1h}--\eqref{eq:var2h}.

\subsection{Abstract error analysis}

We next turn to the a-priori analysis of Galerkin discretizations introduced in the beginning of this section. As usual, see e.g. \cite{Thomee,Varga,Wheeler73}, we decompose the error between the continuous and the semi-discrete solution via 
\begin{align*}
\|u-u_h\|_\V \le \|u - \tilde u_h\|_\V + \|\tilde u_h - u_h\|_\V, \\  
\|p-p_h\|_\Q \le \|p - \tilde p_h\|_\Q + \|\tilde p_h - p_h\|_\Q,  
\end{align*}
into approximation errors and discrete error components. Following \cite{MuradLoula92,MuradLoula94}, we utilize the variational problem corresponding to the stationary system associated with \eqref{eq:var1}--\eqref{eq:var2} to define the approximation $\tilde u_h \in \V_h$ and $\tilde p_h \in \Q_h$, i.e., 
\begin{align}
a(\tilde u_h(t) - u(t), v_h) - b(v_h,\tilde p_h(t) - p(t)) &= 0, \qquad\forall v_h \in \V_h, \ t>0, \label{eq:elliptic1}\\
k(\tilde p_h(t) - p(t), q_h) &= 0, \qquad \forall q_h \in \Q_h, \ t>0. \label{eq:elliptic2}
\end{align}
Note that $\tilde p_h$ and $\tilde u_h$ can be computed  by solving elliptic variational problems and error estimates for the elliptic projection can therefore be obtained by standard arguments; see Section~\ref{sec:biot} below. 
From the definition of the projections and  the discrete solution, we immediately obtain the following discrete error equation.
\begin{lemma} \label{eq:discrete}
Let Assumption~\ref{ass:main} hold and $(u_h,p_h)$ denote a solution of \eqref{eq:var1h}--\eqref{eq:var2h} with initial values $u_h(0)=\tilde u_h(0)$ and $p_h(0)=\tilde p_h(0)$. Then the discrete error components $\delta u_h(t) := u_h(t) - \tilde u_h(t)$ and $\delta p_h(t) := p_h(t) - \tilde p_h(t)$ satisfy 
$\delta u_h(0)=0$, $\delta p_h(0)=0$, and
\begin{alignat}{5}
a(\delta u_h(t),v_h) - b(v_h,\delta p_h(t)) &= 0, \qquad &&\forall v_h \in \V_h, \ t>0, \label{eq:err1h}\\
b(\delta \mathbullet u_h(t), q_h) + k(\delta p_h(t),q_h) &= b(\mathbullet u(t) - \mathbullet{\tilde u}_h(t),q_h), \qquad &&\forall q_h \in \Q_h, \ t>0. \label{eq:err2h} 
\end{alignat}
\end{lemma}
With similar arguments as on the continuous level, we further obtain the following discrete energy--dissipation estimate for the discrete error.
\begin{lemma}
Let Assumption~\ref{ass:main} hold and $(u_h,p_h)$ be a solution of \eqref{eq:var1h}--\eqref{eq:var2h}. Then 
\begin{align*}
\|\delta u_h(t)\|_\V^2 + \int_0^t \|\delta p_h(s)\|_\Q \, ds \le \int_0^t \|B(\mathbullet u(s) - \mathbullet{\tilde u}_h(s))\|^2_{\Q_0^*} \, ds, \qquad 0 \le t \le T.
\end{align*}
If also (A3h) holds, then additionally
\begin{align*}
\beta \|\delta p_h(t)\|_{\Q_0} \le \|\delta u_h(t)\|_\V.
\end{align*}
\end{lemma} 
\begin{proof}
From the discrete error equations \eqref{eq:err1h}--\eqref{eq:err2h}, one can deduce that 
\begin{align*}
\frac{d}{dt} \frac{1}{2} a(\delta u_h,\delta u_h) 
&= a(\delta u_h, \delta \mathbullet u_h) 
 = b(\delta \mathbullet u_h, \delta p_h) 
 = b(\mathbullet u - \mathbullet{\tilde u}_h, \delta p_h) - k(\delta p_h,\delta p_h).
\end{align*}
Using $b(v,q)=\langle B v, q \rangle \le \|Bv\|_{\Q_0^*} \|q\|_{\Q_0}$, the embedding estimate $\|q\|_{\Q_0} \le \|q\|_\Q$, the definition of the norm $\|q\|_\Q^2=k(q,q)$, and Young's inequality allow us to estimate
\begin{align*}
b(\mathbullet u - \mathbullet{\tilde u}_h,\delta p_h)
 \le \frac{1}{2}\|B(\mathbullet u - \mathbullet{\tilde u}_h)\|_{\Q_0^*}^2 + \frac{1}{2} k(\delta p_h, \delta p_h).
\end{align*}
Then integrating with respect to time and using  $\delta u_h(0)=0$  yield 
\begin{align*}
a(\delta u_h(t),\delta u_h(t)) + \int_0^t k(\delta p_h(s),\delta p_h(s)) ds 
\le \int_0^t \|B(\mathbullet u(s) - \mathbullet{\tilde u}_h(s))\|^2_{\Q_0^*} \, ds.
\end{align*}
The first estimate then follows by noting that $\|v\|_\V^2=a(v,v)$ and $\|q\|_\Q^2=k(q,q)$.
From the discrete error equation \eqref{eq:err1h} and continuity of the bilinear form $a$, we deduce that 
\begin{align*}
b(v_h, \delta p_h(t)) = a(\delta u_h(t),v_h) \le \|\delta u_h(t)\|_\V \|v_h\|_\V.
\end{align*}
The discrete inf-sup condition (A3h) then leads to the second estimate of the lemma. 
\end{proof}

\begin{remark}
Let us emphasize that the particular choice of the elliptic projection in the error decomposition 
allows us to bound the discrete error by the approximation error in the component $u$ alone; this leads 
to improved error estimates and allows the use of post-processing techniques to obtain approximations for the pressure in polynomial spaces of higher order; see \cite{BoffiBrezziFortin13,MuradLoula94} for details. 
\end{remark}

Using the previous bounds for the discrete error components, we now obtain the following estimates for semi-discrete approximation \eqref{eq:var1h}--\eqref{eq:var2h}.
\begin{theorem} \label{thm:semi}
Let Assumption~\ref{ass:main} hold. Then 
\begin{align}
\|u(t) - u_h(t)\|_\V^2 &\le \|u(t) - \tilde u_h(t)\|_\V^2 + \int_0^t \|B(\mathbullet u(s) - \mathbullet{\tilde u}_h(s))\|^2_{\Q_0^*} \, ds, \\
\int_0^t \|p(s) - p_h(s)\|^2_\Q \, ds &\le \int_0^t \|p(s) - \tilde p_h(s)\|^2_\Q +  \|B(\mathbullet u(s) - \mathbullet{\tilde u}_h(s))\|^2_{\Q_0^*} \, ds.
\end{align}
If also the discrete inf-sup stability condition (A3h) holds, then additionally
\begin{align}
\|p(t) - p_h(t)\|^2_{\Q_0} & \le \|p(t) - \tilde p_h(t)\|^2_{\Q_0} + \beta_h^{-2} \int_0^t \|B(\mathbullet u(s) - \mathbullet{\tilde u}_h(s))\|^2_{\Q_0^*} \, ds.
\end{align}
\end{theorem}
After choosing the approximation spaces $\V_h$ and $\Q_h$, this estimate allows to derive quantitative error bounds via corresponding estimates for the elliptic projection. Details for a particular discretization will be given in Section~\ref{sec:biot}.

\section{Time discretization} \label{sec:time}

We now turn to the discretization of the semi-discrete variational problem \eqref{eq:var1h}--\eqref{eq:var2h} in time. The guiding principle  will be to preserve the underlying differential-algebraic structure and energy--dissipation identity elaborated in Section~\ref{sec:galerkin} as good as possible. 

\subsection{Preliminaries}

Let $I_\tau=\{0=t^0 < t^1 < \ldots < t^N = T\}$ be a partition of the time interval $[0,T]$ with time steps $\tau_n=t^n-t^{n-1}$ and $\tau=\max_n \tau_n$. 
We denote by $P_q(I_\tau;X) = \{v : v|_{[t^{n-1},t^n]} \in P_q([t^{n-1},t^n];X)\}$ the space of piecewise polynomial functions of $t$ with values in $X$. Following the notation of \cite{Thomee}, we utilize capital letters to denote piecewise polynomial functions of time in the following.

In the spirit of \cite{Akrivis11} and to highlight the preservation of the problem structure, we first give a pointwise definition of our method, which requires the following two projection operators in time: We denote by $\Pi_q^0 : L^2(0,T;X) \to P_{q-1}(I_\tau;X)$ the $L^2$-orthogonal projection and by $\Pi_q^1 : H^1(0,T;X) \to P_q(I_\tau;X) \cap H^1(0,T;X)$, $q \ge 1$, the $H^1$-conforming projection, which is defined by the relations
\begin{align}
\dt \Pi_q^1 u &= \Pi_{q-1}^0 \dt u \quad \qquad \text{and}  \label{eq:proj1}\\
\qquad
(\Pi_q^1 u)(t^{n-1}) &= u(t^{n-1}), \qquad 1 \le n \le N.  \label{eq:proj2}
\end{align}
We will refer to \eqref{eq:proj1} as the \emph{commuting-diagram property}. 
By integration of this relation and use of \eqref{eq:proj2}, one can see that $(\Pi_q^1 u)(t^n) = u(t^n)$ and consequently $\Pi_q^1 u$ is continuous on $[0,T]$; thus $\Pi_q^1$ is an $H^1$-conforming projection in time.
We now consider the following time-discretization of the semi-discrete problem. 
\begin{problem}[Fully discrete scheme] \label{prob:fullh}
Find $U_h \in P_q(I_\tau;\V_h) \cap H^1(0,T;\V_h)$ and $P_h \in P_q(I_\tau;\Q_h) \cap H^1(0,T;\Q_h)$ such that 
$U_h(0)=\tilde u_h(0)$, $P_h(0)=\tilde p_h(0)$, and  
\begin{alignat}{4}
a(U_h(t),v_h) - b(v_h, P_h(t)) &= \langle \Pi_q^1 f(t), v_h\rangle, \qquad &&\forall v_h \in \V_h, \ t>0, \label{eq:fullh1}\\
b(\mathbullet U_h(t), q_h) + k(\Pi_{q-1}^0 P_h(t), q_h) &= \langle \Pi_{q-1}^0 g(t), q_h\rangle, \qquad &&\forall q_h \in \Q_h, \ t>0. \label{eq:fullh2}
\end{alignat}
\end{problem}
Using similar arguments as already employed for the analysis on the continuous and the semi-discrete level, one can show the following energy--dissipation identity. 
\begin{lemma} \label{lem:wellposedfullh}
Let Assumption~\ref{ass:main} and (A3h) hold. Then Problem~\ref{prob:fullh} has a unique solution
and the following energy identity is valid for $0 \le n \le N$:
\begin{align*}
\frac{1}{2} a(U_h(t^n),U_h(t^n)) &+ \int_0^{t^n} k(\Pi_{q-1}^0 P_h(t),\Pi_{q-1}^0 P_h(t)) \, dt \\
&= \frac{1}{2} a(U_h(0),U_h(0)) + \int_0^{t^n} \langle \Pi_q^1 f(t), \mathbullet U_h(t)\rangle + \langle \Pi_{q-1}^0 g(t), P_h(t)\rangle \, dt.
\end{align*}
\end{lemma}

The proof of the lemma will be presented in the following subsection.  
After that, we will give an interpretation of the fully discrete scheme as a continuous Galerkin approximation or as a variant of particular Runge-Kutta methods for the modified system arising after semi-discretization and differentiation of the algebraic equation; compare with \eqref{eq:daesys2}. 
In the last part of this section, we  present a detailed a-priori error analysis. 

\subsection{Proof of Lemma~\ref{lem:wellposedfullh}}

We start with proving the energy estimate. 
To this end, let us denote by $(U_h, P_h)$ a solution of Problem~\ref{prob:fullh}. Then 
\begin{align*}
\frac{1}{2} a(U_h(t^n),U_h(t^n)) &- \frac{1}{2} a(U_h(t^{n-1}),U_h(t^{n-1}))
 = \int_{t^{n-1}}^{t^n} a(U_h(t),\mathbullet U_h(t)) \, dt \\
&= \int_{t^{n-1}}^{t^n} \langle \Pi_q^1 f(t), \mathbullet U_h(t)\rangle + b(\mathbullet U_h(t), P_h(t)) dt \\
&= \int_{t^{n-1}}^{t^n} \langle \Pi_q^1 f(t), \mathbullet U_h(t)\rangle + \langle \Pi_{q-1}^0 g(t),  P_h(t)\rangle - k(\Pi_{q-1}^0 P_h(t), P_h(t)) \, dt.
\end{align*}
The energy identity now follows by noting that 
\begin{align*}
\int_{t^{n-1}}^{t^n} k(\Pi_{q-1}^0 P_h(t), P_h(t)) \,  dt =   \int_{t^{n-1}}^{t^n} k(\Pi_{q-1}^0 P_h(t), \Pi_{q-1}^0 P_h(t)) \, dt.
\end{align*}

As a next step, we show uniqueness. Due to linearity of the problem, it suffices to verify that $f \equiv 0$, $g \equiv 0$ and $U_h(0) = 0$, $P_h(0) = 0$ imply $U_h \equiv 0$,  $P_h \equiv 0$. For homogeneous data, we can deduce from the energy-identity that 
\begin{align*}
\frac{1}{2} \|U_h(t^n)\|_\V^2 + \int_{t^{n-1}}^{t^n} \|\Pi_{q-1}^0 P_h(t)\|^2_\Q \, dt = 0.
\end{align*}
This implies $U_h(t^n)=0$ and $\Pi_{q-1}^0 P_h(t)=0$ for all $t>0$. Using (A3h) and the  equation \eqref{eq:fullh1}, we further obtain $P_h(t^n)=0$; the latter two conditions imply $P_h \equiv 0$. From condition (A1) and equation \eqref{eq:fullh1}, we further conclude that $U_h \equiv 0$.

To establish existence of a solution, we proceed as follows: 
After choosing a basis for $\V_h$, $\Q_h$, we can identify $U_h(t)$, $P_h(t)$ with vectors $U(t)$, $P(t)$, which are continuous, piecewise polynomial functions of time. On the interval $[t^{n-1},t^n]$, they can be expressed as $U(t) = U(t^{n-1}) + \sum_{i=1}^{q+1} U^n_i (t-t^{n-1})^i$ and  $P(t)=P(t^{n-1}) + \sum_{i=1}^{q+1} P^n_i (t-t^{n-1})^i$, respectively. Evaluating \eqref{eq:fullh1}--\eqref{eq:fullh2} at distinct time points $t^n_i \in (t^{n-1},t^n]$, $i=1,\ldots,q+1$, leads to a linear system with the same number of unknowns and equations. 
From the energy identity and the previous considerations, we already know that the solution is unique which, in finite dimensions, then also guarantees the existence of a solution. 

\subsection{Relation to other time-discretization schemes} \label{sec:timerelation}

The following considerations allow us to interpret the fully discrete method as a continuous-Galerkin approximation or as a variant of a Runge-Kutta method. 

By differentiating \eqref{eq:fullh1} in time, 
the system \eqref{eq:fullh1}--\eqref{eq:fullh2} can be seen to be equivalent to the variational equations
\begin{alignat}{4}
\int_0^T a(\mathbullet U_h(t),v_h(t)) - b(v_h(t), \mathbullet P_h(t)) \, dt  &= \int_0^T \langle \Pi_{q-1}^0 \mathbullet f(t), v_h(t)\rangle \, dt,  \label{eq:cg1}\\
\int_0^T b(\mathbullet U_h(t), q_h(t)) + k(\Pi_{q-1}^0 P_h(t), q_h(t)) \, dt &= \int_0^T \langle \Pi_{q-1}^0 g(t), q_h(t)\rangle \, dt, \label{eq:cg2}
\end{alignat}
which hold for all space-time test functions $v_h \in P_{q-1}(I_\tau;\V_h)$ and $q_h \in P_{q-1}(I_\tau;\Q_h)$.
Let us note that the projections on the right hand side of \eqref{eq:cg1}--\eqref{eq:cg2} could be dropped.
This shows that Problem~\ref{prob:fullh} coincides with a continuous-Galerkin (Petrov-Galerkin) time discretization of the modified system \eqref{eq:daesys2} arising after semi-discretization in space and differentiation of the algebraic equation.

Now let $t^n_i$ and $b_i^n$, $i=1,\ldots,q+1$, denote the Gau\ss-Lobatto quadrature points and weights on the interval $[t^{n-1},t^n]$ and recall that $\int_{t^{n-1}}^{t^n} p(t) dt = \sum_{i=1}^{q+1}  p(t^n_i) b_i^n$ for all polynomials $p \in P_{2q+1}(t^{n-1},t^n)$. 
Then \eqref{eq:cg1}--\eqref{eq:cg2} can be rewritten equivalently as 
\begin{alignat}{4}
a(\mathbullet U_h(t^n_i),v_h) - b(v_h, \mathbullet P_h(t^n_i))  &= \langle \Pi_{q-1}^0 \mathbullet f(t^n_i), v_h\rangle, \label{eq:gauss1}\\
b(\mathbullet U_h(t^n_i), q_h) + k(\Pi_{q-1}^0 P_h(t^n_i), q_h) &= \langle \Pi_{q-1}^0 g(t^n_i), q_h\rangle,  \label{eq:gauss2}
\quad 0 \le i \le q+1,
\end{alignat}
for all time steps $1 \le n \le N$. Hence \eqref{eq:cg1}--\eqref{eq:cg2} can also be interpreted as the Lobatto-IIIA  Runga-Kutta collocation method with approximation of the right hand sides by appropriate projections. Alternatively, the method could be interpreted as an inexact realization of the Gau\ss-Runga-Kutta method of appropriate order; see \cite{Akrivis11} for the discussion of the close relation between Galerkin and Runge-Kutta time discretization schemes. 

\begin{remark}
In summary, Problem~\ref{prob:fullh} can be interpreted as a continuous-Galerkin or inexact Runge-Kutta method applied to the modified system \eqref{eq:daesys2} arising after semi-discretization and differentiation of the algebraic equation.
While the original pointwise form \eqref{eq:fullh1}--\eqref{eq:fullh2} will be advantageous for the numerical analysis, the interpretation as a Runge-Kutta method can serve as the basis for the actual implementation.
\end{remark}

\subsection{Error analysis}

For the error analysis of the fully-discrete scheme, we proceed similar to the semi-discrete level and utilize an error decomposition 
\begin{align*}
\|u - U_h\|_\V &\le \|u-\Pi_q^1 \tilde u_h\|_\V + \|\Pi_q^1 \tilde u_h - U_h\|_\V \\
\|p - P_h\|_\Q &\le \|p-\Pi_q^1 \tilde p_h\|_\Q + \|\Pi_q^1 \tilde p_h - P_h\|_\Q. 
\end{align*}
Estimates for the projection errors in space and time can be obtained with standard arguments. In the following, we therefore only consider the discrete error components.

\begin{lemma} 
Let $\delta U_h = \Pi_q^1 \tilde u_h - U_h$ and $\delta P_h = \Pi_q^1 \tilde p_h - P_h$ be the discrete errors. Then
$\delta U_h(0)=0$, $\delta P_h(0)=0$, and 
for all $v_h \in \V_h$, $q_h \in \Q_h$, and a.e. $0 \le t \le T$,
\begin{alignat}{4}
a(\delta U_h(t),v_h) - b(v_h, \delta P_h(t)) &= 0, \label{eq:fullerrh1}\\
b(\delta \mathbullet U_h(t), q_h) + k(\Pi_{q-1}^0 \delta P_h(t), q_h) &= b(\Pi_{q-1}^0 (\mathbullet u - \mathbullet{\tilde u}_h)(t), q_h) \label{eq:fullerrh2} \\ 
&\qquad + k(\Pi_{q-1}^0 (p -  \tilde p_h)(t), q_h). \notag
\end{alignat}
\end{lemma}
\begin{proof}
The identities are a direct consequence of the properties of the elliptic projections and 
the commuting diagram property of the time projection operators. 
\end{proof}

Using the energy-dissipation identity of the discrete problem stated in Lemma~\ref{lem:wellposedfullh} now allows us to obtain the following estimates for the discrete error components.
\begin{lemma}
Under the assumptions of the previous lemmas, there holds 
\begin{align*}
\|\delta U_h(t^n)\|^2_\V + \int_0^{t^n} \|\Pi_{q-1}^0 \delta P_h(t)\|_\Q^2 \, dt 
\le 2 \int_0^{t^n} \|B (\mathbullet u - \mathbullet{\tilde u}_h)\|^2_{\Q_0^*}  + \|p - \Pi_q^1 p\|^2_\Q \, dt 
\end{align*} 
\end{lemma}
\begin{proof}
Using the energy-dissipation identity of Lemma~\ref{lem:wellposedfullh} for the system \eqref{eq:fullerrh1}--\eqref{eq:fullerrh2} and further employing $\delta U_h(0)=0$ and $\delta P_h(0)=0$ yield
\begin{align*}
\frac{1}{2} a(\delta U_h(t^n),&\delta U_h(t^n)) +\int_0^{t_n} k(\Pi_{q-1}^0 \delta P_h(t),\Pi_{q-1}^0\delta P_h(t)) dt\\
&= \int_0^{t^n} b(\Pi_{q-1}^0 (\mathbullet u - \mathbullet{\tilde u}_h), \Pi_{q-1}^0 \delta P_h) + k(\Pi_{q-1}^0 (p - \tilde p_h),\Pi_{q-1}^0 \delta P_h)dt=(*).
\end{align*}
By the Cauchy-Schwarz and Young's inequality, the boundedness of the $L^2$-projection, 
and using $\|q\|_{\Q_0} \le \|q\|_\Q$, we can estimate the two terms in the last line by 
\begin{align*}
(*) \le \int_0^{t^n} \|B(\mathbullet u - \mathbullet{\tilde u}_h)\|_{\Q_0^*}^2 + \|p - \Pi_q^1 p\|^2_{\Q}  + \frac{1}{2} \|\Pi_{q-1}^0 \delta P_h\|_\Q^2 \, dt 
\end{align*}
By definition of the norm, we have $\|q\|_\Q^2=k(q,q)$, and the last term in the above estimate can be absorbed by the left hand side in the energy identity, which already yields the assertion of the lemma. 
\end{proof}

Using the error decomposition stated above and the bounds for the discrete error components, we arrive at the following abstract error estimates. 
\begin{theorem} \label{thm:full}
Let Assumption~\ref{ass:main} hold. Then 
\begin{align*}
\|u(t^n) - U_h(t^n)\|_\V &\le \|u(t^n) - \tilde u_h(t^n)\|_\V + C_h(u) + C_\tau(p) \\  
\|\Pi_{q-1}^0(p - P_h)\|_{L^2(0,t^n;\Q)} &\le \|\Pi_{q-1}^0(p - \tilde p_h)\|_{L^2(0,t^n;\Q)} + C_h(u) + C_\tau(p)
\end{align*}
with projection errors
$C_h(u)=\|B (\mathbullet u - \mathbullet{\tilde u}_h)\|_{L^2(0,t^n;\Q_0^*)}$ and $C_\tau(p)=\|p - \Pi_q^1 p\|_{L^2(0,t^n;\Q)}$. If, in addition, also (A3h) holds, then 
\begin{align*}
\|p(t^n) - P_h(t^n)\|_{\Q_0} \le \|p(t^n) - \tilde p_h(t^n)\|_{\Q_0} + \beta_h^{-1}\big( C_h(u) + C_\tau(p) \big).  
\end{align*}
\end{theorem}

The abstract error estimates given above allow us to analyse a large class of Galerkin approximations in space 
and time discretization schemes of arbitrary order. In the following section, we discuss one 
particular discretization for the two-field formulation of the Biot system \eqref{eq:biot1}--\eqref{eq:biot2} 
and we establish explicit high-order estimates in space and time. 

\section{Application to the Biot system} \label{sec:biot}

We now apply the abstract discretization framework of the previous sections to a particular discretization of the Biot system 
\begin{align}
-\div (2 \mu \eps(u) + \lambda \div(u) I) + \alpha \nabla p  &=  f, \label{eq:biot1num}\\
\alpha \div (\mathbullet u) - \div(\kappa \nabla p) &= g, \label{eq:biot2num}
\end{align}
over some bounded polyhedral Lipschitz domain $\Omega$ and a finite time interval $[0,T]$.  
For ease of presentation, we consider homogeneous boundary conditions
\begin{align} \label{eq:biot3num}
u = 0 \qquad \text{and} \qquad p = 0 \qquad \text{on } \partial\Omega.
\end{align}
The natural function spaces for the problem \eqref{eq:biot1num}--\eqref{eq:biot3num} are then given by
$$
\V=H_0^1(\Omega)^d, \qquad \Q=H_0^1(\Omega), \qquad \text{and} \qquad \Q_0=L^2(\Omega). 
$$
We further assume that the model parameters $\mu$, $\lambda$, $\kappa$ are smooth functions and uniformly bounded from above and below, and $\alpha$ is a positive constant. The validity of conditions (A1)--(A2) follows from the Friedrichs' inequality, and (A3) corresponds to the usual inf-sup condition of incompressible flow; see e.g. \cite{GiraultRaviart86}. 

Now let $\Th$ be a shape-regular conforming simplicial mesh of the domain $\Omega$ and denote by $P_k(\Th)$ the space of piecewise polynomials of degree $\le k$ over the mesh $\Th$; see \cite{ErnGuermond}. 
For the space discretization, we choose the Taylor-Hood elements and set
\begin{align}
V_h = P_{k+1}(\Th) \cap H_0^1(\Omega)^d \qquad \text{and} \qquad Q_h = P_k(\Th) \cap H_0^1(\Omega). 
\end{align}
It is well-known \cite{BoffiBrezziFortin13} that these spaces satisfy the discrete inf-sup condition (A3h). 
Therefore, all estimates of Theorems~\ref{thm:semi} and \ref{thm:full} can be applied. 
In order to obtain quantitative estimates, we require bounds for the spatial and temporal projection errors. 
By standard interpolation estimates, we obtain the following result for the projection in time.
\begin{lemma} 
Let $\Pi_q^1$ denote the projection operator defined in \eqref{eq:proj1}--\eqref{eq:proj2}. Then 
\begin{align*}
\|v - \Pi_q^1 v\|_{L^2(t^{n-1},t^n;X)} \le C \tau_n^r \|\dt^{(r)} v\|_{L^2(t^{n-1},t^n;X)}, \quad 1 \le r \le q+1, 
\end{align*}
for any piecewise smooth function $v \in H^{r+1}(I_\tau;X)$ in time.
\end{lemma}
By combination of standard finite element error estimates, we further obtain the following bounds for the elliptic projection defined in \eqref{eq:elliptic1}--\eqref{eq:elliptic2}. 
\begin{lemma}
Let (A1)--(A3) hold and $V_h$, $Q_h$ be defined as above. Then 
\begin{align*}
\|p - \tilde p_h\|_{H^1(\Omega)} \le C h^{s} \|p\|_{H^{s+1}(\Th)}, \qquad 0 \le s\le k, \\
\|u - \tilde u_h\|_{H^1(\Omega)} \le C h^{s} (\|u\|_{H^{s+1}(\Th)} + \|p\|_{H^{s+1}(\Th)}), \qquad 0 \le s \le k,
\end{align*}
for any piecewise smooth $p \in H_0^1(\Omega) \cap H^{s+1}(\Th)$ and $u \in H_0^1(\Omega)^d \cap H^{s+1}(\Th)^d$. \\
If, additionally, $\Omega$ is convex and $u \in H^{s+2}(\Th)^d$, then
\begin{align*}
\|p - \tilde p_h\|_{L^2(\Omega)} \le C h^{s+1} \|p\|_{H^{s+1}(\Th)}, \qquad 0 \le s \le k, \\
\|u - \tilde u_h\|_{H^1(\Omega)} \le C h^{s+1} (\|u\|_{H^{s+2}(\Th)} + \|p\|_{H^{s+1}(\Th)}), \qquad 0 \le s \le k.
\end{align*}
\end{lemma}
\begin{proof}
The bounds for the pressure component follow directly from standard error estimates for the elliptic problem \eqref{eq:elliptic2}; see e.g. \cite{ErnGuermond}. 
From \eqref{eq:elliptic1}, one can see that
\begin{align*}
 a(u - \tilde u_h,v_h) &= b(v_h,p - \tilde p_h).
\end{align*}
Using the boundedness of $a$ and $b$, and the ellipticity of $a$, we obtain 
\begin{align*}
\|u-\tilde u_h\|_{H^1}^2 
&= a(u-\tilde u_h,u-\tilde u_h)   \\
&= a(u-\tilde u_h,u-v_h) + b(\tilde u_h-v_h,p-\tilde p_h) \\
&\le \|u-\tilde u_h\|_{H^1} \|u-v_h\|_{H^1} + (\|u - \tilde u_h\|_{H^1} + \|u - v_h\|_{H^1}) \|p - \tilde p_h\|_{L^2},
\end{align*}
for all $v_h \in V_h$.
By Young's inequality and rearrangement of terms, this leads to
\begin{align*}
\|u-\tilde u_h\|_{H^1}^2 \le C (\|u-v_h\|^2_{H^1} + \|p-\tilde p_h\|_{L^2}^2).  
\end{align*}
The two estimates for the displacement error now follow from those for the pressure and standard approximation error estimates for the space $V_h$.
\end{proof}
From the estimates of Theorem~\ref{thm:semi}, we now immediately deduce the following result.
\begin{theorem}
Let Assumption~\ref{ass:main} hold and $(u,p)$ denote a sufficiently regular weak solution of the system \eqref{eq:biot1num}--\eqref{eq:biot3num} with uniformly positive smooth functions $\lambda,\mu,\kappa$, and $\Omega$ convex. Moreover, let $u_h,p_h$ denote a solution of \eqref{eq:var1h}--\eqref{eq:var2h} with $V_h$, $Q_h$ chosen as above. Then
\begin{align*}
\|u-u_h\|_{L^\infty(0,T;H^1(\Omega))} + \|p-p_h\|_{L^\infty(0,T;L^2(\Omega))} + h \|p - p_h\|_{L^2(0,T;H^1(\Omega))} \le C(u,p) h^{k+1}  
\end{align*}
with a constant $C(u,p)$ depending only on the norm of the solution. 
\end{theorem}
Similar estimates were obtained by Murad and Loula \cite{MuradLoula94} via different energy arguments.
Using Theorem~\ref{thm:full}, we further obtain the following fully discrete error estimates.
\begin{theorem}
Let the assumptions of the previous theorem hold and 
$(U_h,P_h)$ denote the fully discrete solution defined in Problem~\ref{prob:fullh}. Then 
\begin{align*}
&\max_{0 \le t^n \le T} \|u(t^n)-U_h^n\|_{H^1(\Omega)} + \max_{0 \le t^n \le T} \|p(t^n)-P_h^n\|_{L^2(\Omega)} 
\\&\qquad \qquad \qquad 
+ h \big(\sum_{n=1}^N \tau \|p(t^n) - P_h^n\|_{H^1(\Omega)}^2 )^{1/2}
\le C_1(u,p) h^{k+1} + C_2(u,p) \tau^{q+1}
\end{align*}
with constants $C_i(u,p)$, $i=1,2$ depending only on the norm of the solution. 
\end{theorem}
Let us note that only first order estimates with respect to  the time discretization were obtained in \cite{KanschatRiviere18,MuradLoula94}, and the results of the previous theorem seem to be the first rigorous high-order estimates in space and time.

\section{Numerical tests} \label{sec:num}

We now illustrate our theoretical findings by numerical results for a test problem which was utilized in \cite{KanschatRiviere18}. 
We consider the Biot-system \eqref{eq:biot1}--\eqref{eq:biot2} with constant parameters $\alpha=\mu=\lambda=\kappa=1$ on the two-dimensional unit square $\Omega=(0,1)^2$. 
As in \cite{KanschatRiviere18}, the exact solution shall be given by   
\begin{align*}
  p(x,y,t)&=\psi(t)\phi(x,y), \\
  u(x,y,t)&=\frac{\psi(t)}{8 \pi^2}\nabla\phi(x,y),
\end{align*}
with $\phi(x,y)=\sin(2\pi x) \sin(2 \pi y)$ and time dependent function
\begin{align*}
\psi(t)=\frac{1}{64\pi^4+4\pi^2}(8\pi^2\sin(2\pi t)-2\pi \cos(2 \pi t)+2 \pi e^{-8\pi^2 t}). 
\end{align*}
This solution satisfies the somewhat non-standard boundary conditions
\begin{align*}
n \times u = 0, \qquad 
\partial_n (u \cdot n) = 0, \qquad 
p=0 \qquad \text{on } \partial \Omega.
\end{align*}
One can verify that the first two conditions amount to mixed Dirichlet and Neumann boundary conditions for any of the two components of the deformation $u$. 
The problem data $f$ and $g$ are determined by inserting the exact solution into the Biot equations. 
Let us note that our abstract convergence results immediately apply to this problem. Moreover, since the solution is smooth, we expect to observe the full convergence rates predicted by our theoretical results.

\subsection{Remarks on the implementation}

As indicated in Section~\ref{sec:timerelation}, the proposed time discretization strategy can be interpreted as a variant of the Lobatto-IIIA method with $s=q+1$ stages applied to the integration of the modified differential-algebraic system \eqref{eq:daesys2} which arises after differentiation of the algebraic equation. 
For the lowest order approximation $q=1$, the resulting scheme can be written as 
\begin{align*}
\begin{pmatrix} A & -B^\top \\ B & 0\end{pmatrix} 
\begin{pmatrix} \mathbullet u^{n+1/2} \\ \mathbullet p^{n+1/2} \end{pmatrix}
+ 
\begin{pmatrix} 0 & 0 \\ 0 & K\end{pmatrix} 
\begin{pmatrix} u^{n+1/2} \\ p^{n+1/2} \end{pmatrix}
=
\begin{pmatrix} \tilde{\mathbullet f}^{n+1/2} \\ \tilde g^{n+1/2} \end{pmatrix}
\end{align*}
where $\mathbullet a^{n+1/2} = \frac{1}{\tau} (a^{n+1}-a^n)$, $a^{n+1/2} = \frac{1}{2}(a^{n+1}+a^n)$, and $\tilde a^{n+1/2} = \frac{1}{\tau}\int_{t^n}^{t^{n+1}} a(t) dt$.
Apart from the special form of the right hand sides, this corresponds to the 
Crank-Nicolson method, i.e., the Lobatto-IIIA method with $s=2$ stages. 
Due to the stability of problem \eqref{eq:daesys2} with respect to the data,
one can use numerical quadrature for the right hand sides without disturbing the convergence rate. 
In our numerical tests, we will therefore use the Lobatto-IIIA method with $s=q+1$ stages for the time integration of the modified system \eqref{eq:daesys} instead of the Petrov-Galerkin approximation with order $q$.
For discretization of the domain $\Omega$, we utilize uniform triangulations obtained by regular refinements of an initial mesh consisting of only two triangles. Taylor-Hood finite elements $P_{k+1}-P_k$ of order $k$ are utilized for the spatial approximation of the funtions $u$ and $p$, as discussed in Section~\ref{sec:biot}.

\subsection{Results}

In our first test, we consider the approximation by $P_{2}-P_1$ elements in space and the Crank-Nicolson method in time. This corresponds to polynomial orders $k=q=1$ in the theorems presented in the previous section.
The results of our computations are summarized in Table~\ref{tab:1},
where we display relative errors $\|y-y_h\|_{rel} = \|y-y_h\|/\|y\|$ and approximate norms in time by evaluations at the discrete time points $t^n=n \tau$. 

\begin{table}[hbt!]
\begin{center}
\bgroup
\footnotesize
\def\arraystretch{1.1}
\begin{tabular}{c||c|c||c|c||c|c}
 $h$ & $\|u-u_h\|_{L_{\tau,rel}^\infty(H^1)}$  & eoc 
     & $\|p-p_h\|_{L_{\tau,rel}^\infty(L^2)}$  & eoc 
     & $\|p-p_h\|_{L_{\tau,rel}^2(H^1)}$       & eoc \\[3pt]
\hline  
1/8   & 1.5374e-01 & ---  & 2.5105e-01 & ---  & 3.8562e-01 & ---  \\
1/16  & 4.2186e-02 & 1.87 & 7.1120e-02 & 1.82 & 1.9495e-01 & 0.98 \\
1/32  & 1.0808e-02 & 1.96 & 1.8365e-02 & 1.95 & 9.7553e-02 & 1.00 \\
1/64  & 2.7189e-03 & 1.99 & 4.6288e-03 & 1.99 & 4.8779e-02 & 1.00 \\                          
\end{tabular}
\egroup
\vskip1em
\caption{Relative errors for approximation with $P_2$--$P_1$ finite elements in space and the Crank-Nicolson method in time with $\tau=0.1 h$.\label{tab:1}}  
\end{center}
\end{table}

As predicted by the theorem of Section~\ref{sec:biot}, we can observe second order convergence in the $H^1$-norm for the displacement and the $L^2$-norm for the pressure when choosing the time step $\tau = c h$ proportional to the mesh size. Due to the lower polynomial order of the approximation, the $H^1$-error in the pressure is limited to one. 

\bigskip 

In order to illustrate the possibility for higher-order approximations, we consider in a second test the spatial approximation by $P_4$-$P_3$ elements   
together with time discretization via the Lobatto-IIIA method with $s=3$ stages. 
We again choose the time step $\tau=0.1 h$ proportional to the mesh size. 
The corresponding results are depicted in Table \ref{tab:2}. 
\begin{table}[hbt!]
\begin{center}
\bgroup
\def\arraystretch{1.1}
\small
\begin{tabular}{c||c|c||c|c||c|c}
 $h$ & $\|u-u_h\|_{L_{rel}^\infty(H^1)}$  & eoc 
     & $\|p-p_h\|_{L_{rel}^\infty(L^2)}$  & eoc 
     & $\|p-p_h\|_{L_{rel}^2(H^1)}$       & eoc \\[3pt]
\hline  
1/8  &  7.7344e-04 & ---  & 6.8360e-04 & ---  & 5.8759e-03 & ---  \\
1/16 &  4.9170e-05 & 3.98 & 4.1778e-05 & 4.03 & 7.3638e-04 & 3.00 \\
1/32 &  3.0855e-06 & 3.99 & 2.5781e-06 & 4.02 & 9.1886e-05 & 3.00 \\
1/64 &  1.9299e-07 & 4.00 & 1.6018e-07 & 4.01 & 1.1470e-05 & 3.00 
\end{tabular}
\egroup
\vskip1em
\caption{Relative errors for $P_4$--$P_3$ finite elements in space and the Lobatto-IIIA method with $s=3$ stages in time and $\tau=0.1 h$.\label{tab:2}}  
\end{center}
\end{table}

The approximation utilized for our computations corresponds the setting discussed in Section~\ref{sec:biot} with polynomial orders $k=3$ and $q=2$. 
The convergence rate for the $H^1$-error in the pressure is explained by our theoretical results. For the $H^1$-error in the displacement and $L^2$-error in the pressure, we proved error bounds of the form $O(h^4 + \tau^3)$. The results obtained in our computations thus seem to illustrate super-convergence $O(\tau^4)$ with respect to the time discretization at discrete time points $t^n=n \tau$, which is the rate that can be expected for the Lobatt-IIIA methods with $s=3$ stages or the Petrov-Galerkin approximation with order $q=2$, when applied to the solution of ordinary differential equations; see \cite{Akrivis11} for details. A rigorous proof of this super-convergence in the context of space-time discretization of the Biot system is still open.

\section{Discussion}

We considered the systematic  approximation of a class of abstract evolution problems by Galerkin methods in space and time. This class of problems covers the quasistatic Biot-system as a special case, which allowed us to derive convergence rates for high-order approximations by inf-sup stable finite elements in space and variants of Runge-Kutta methods in time applied to a certain reformulation of the problem. 
The predicted rates were confirmed in numerical tests and super-convergence with respect to the time discretization could be observed at discrete time points. A rigorous analysis of this fact is a possible topic for future research. 
In this paper, we considered abstract evolution problems whose strucure was motivated by the two-field formulation of the Biot system. The main arguments of our analysis, however, seem applicable also to other formulations of the problem and also to other time-discretization schemes, e.g., discontinuous-Galerkin methods or Runge-Kutta methods of Radau-type,
which have stronger stability properties.

{\small 
\section*{Acknowledgements}
This work was supported by the German Research Foundation (DFG) via grants TRR~146 C3, TRR~154 C4, Eg-331/1-1. The second author was additionally supported by the ``Center for Computational Engineering'' and the ``Future-Talents'' program at TU Darmstadt. 
The authors would further like to thank Prof. Johannes Kraus from University Duisburg-Essen for interesting discussions on quasistatic poroelasticity which initiated this research. 
}


\end{document}